\documentclass{article}

\usepackage{arxiv}

\usepackage[utf8]{inputenc} 
\usepackage[T1]{fontenc}    
\usepackage{hyperref}       
\usepackage{url}            
\usepackage{booktabs}       
\usepackage{amsfonts}       
\usepackage{nicefrac}       
\usepackage{microtype}      
\usepackage{lipsum}
\usepackage{graphicx}
\graphicspath{ {./images/} }
\usepackage{amsthm, amsfonts, amsmath, amssymb}

\newtheorem{theorem}{Theorem}[section]
\newtheorem{proposition}[theorem]{Proposition}
\newtheorem{corollary}[theorem]{Corollary}

\newtheorem{definition}[theorem]{Definition}
\newtheorem{example}[theorem]{Example}
\newtheorem{remark}[theorem]{Remark}

\usepackage[overload]{empheq}
\usepackage{nccmath}

\usepackage{quiver}
\usepackage{tikz-cd}
\usepackage{ dsfont }
\usepackage{mathtools}
\usepackage{ mathrsfs }

\usepackage{tikz}
\usetikzlibrary{matrix}
\usetikzlibrary{calc}
\usepackage{blindtext}
\usepackage{authblk} 

\title{Expanding \v{C}ech cohomology for quantales}

\author[1]{Ana Luiza Tenório}
\author[2]{Peter Arndt}
\author[3]{Hugo Luiz Mariano}
\affil[1]{Department of Applied Mathematics, University of Rio de Janeiro, \texttt{anatenorio@im.ufrj.br}}
\affil[2]{Heinrich-Heine-Universität Düsseldorf, \texttt{Peter.Arndt@hhu.de}}
\affil[3]{Department of Mathematics, University of São Paulo, \texttt{hugomar@ime.usp.br}}

\begin{document}
\maketitle
\begin{abstract}{
We expand \v{C}ech cohomology of a topological space $X$ with values in a presheaf on $X$ to \v{C}ech cohomology of a commutative ring with unity $R$ with values in a presheaf on $R$. The strategy is to observe that both the set of open subsets of $X$ and the set of ideals of $R$ provide examples of a (semicartesian) quantale. We study a particular pair of (adjoint) functors $(\theta, \tau)$ between the quantale of open subsets of $X$ and the quantale of ideals of $C(X)$, the ring of real-valued continuous functions on $X$. This leads to the main result of this paper: the $q$th \v{C}ech cohomology groups of $X$ with values on the constant sheaf $F$ on $X$ is isomorphic to the $q$th \v{C}ech cohomology groups of the ring $C(X)$ with values on a sheaf $F \circ \tau$ on $C(X)$.  
}
\end{abstract}


\section{Introduction}
In this paper, we present a simple yet powerful method to answer the following question:\textit{ how we define \v{C}ech cohomology for $R$ a commutative ring with unity with values in a (pre)sheaf on $R$?} The method is simple because we use \v{C}ech cohomology for a topological space $X$ with values in a presheaf in $X$ and observe that both the open subsets of $X$ and the ideals of $R$ form a structure called quantale -- a complete lattice equipped with a multiplication that distributes over arbitrary joins, then generalizing the situation when binary meets distribute over arbitrary joins, which gives \textit{locales}. In this way, we define \v{C}ech cohomology for an element in a quantale in a way that we are just replacing the intersections between open subsets of $X$ by a more general binary operation that encompass the multiplication of ideals of $R$. In symbols, the \v{C}ech cochain complex in degree $q$ for a cover of $\{U_i\}_{i \in I}$ in a quantale $Q$ with values in a presheaf $F$ is given by 
$$ \prod_{i_0<...<i_q}F(U_{i_0} \odot \dots \odot U_{i_q}),$$
where $\odot$ is the multiplication of the quantale.

The significance of the method relies in its ability to explore connections between the \v{C}ech cohomology groups of different quantales. We have found an example of this phenomenon by defining a morphism from the quantale of open subsets of a topological space $X$ to the quantale of the ideals of the ring $C(X)$  of continuous real valued functions on $X$ and then proving (Theorem \ref{teo:iso_of_cohomology_of_space}) that the $q$th \v{C}ech cohomology group for $X$ with values in a constant sheaf on $X$ is isomorphic to the $q$th \v{C}ech cohomology group for $C(X)$ with values in constant sheaf on $C(X)$. We argue that such isomorphism is a tool to relate topological properties of $X$ with algebraic properties of $C(X)$. We have also identified properties between arbitrary (semicartesian and commutative) quantales that will induce an isomorphism, for each degree, in the cohomological level (Theorem \ref{teo:iso_cohomology_any_quantale}).  

Our starting question is relevant because it may provide a new technique to study rings. On one hand recall that $X$ is connected if and only if $C(X)$ only has trivial idempotent elements. On the other hand, recall that the dimension of the de Rham cohomology $H^0_{dR}(X)$, which is isomorphic to the \v{C}ech cohomology $\check{H}^0(X, \mathbb R)$ if $X$ is a smooth manifold, corresponds to the  number of connected components of $X$. Therefore, we have an indication that the \v{C}ech cohomology group for $R$ in degree $0$ is providing the existence of non-trivial idempotent elements in $R$, perhaps, even counting them. We hope that further developments of \v{C}ech cohomology for quantales will bring novel algebro-geometric results and provide methods to study commutative and non-commutative rings.

\textit{Structure of the paper:} In Section \ref{Sec:sheaves_rings}, we introduce sheaves on quantales as defined in \cite{luiza2024sheaves}. In Section\ref{sec:base_change}, we present a morphism of quantales that we call \textit{geometric morphism}, based on the nomenclature for sheaves on locales. This is important to discuss a process of base change and, in particular, to obtain sheaves on quantales from well-known sheaves on locales. The morphism between open subsets of $X$ and ideals of $C(X)$ is also developed in this section. Finally, the last section is devoted to the \v{C}ech cohomology for quantales and its interpretations.

\begin{remark}
    \v{C}ech cohomology is also used for a site $(C,J)$ where $C$ is a category with pullbacks and $J$ is a Grothendieck (pre)topology with values in a sheaf for that site. In this framework, the intersection of open subsets of a topological space is replaced by pullback of objects that form a cover in the sense of a Grothendieck pretopology. In \cite{luiza2024sheaves} we proved that the category of sheaves on a quantale is not always a topos, therefore, the notion of cover is not compatible with the one provided by Grothendieck pretopologies. In \cite{tenório2024grothendieck}, we provided \textbf{Grothendieck prelopologies} as a generalization of Grothendieck pretolopologies in a way that the considered notion of cover in a quantale is an example of a Grothendieck prelopology. 
\end{remark}

\section{Sheaves on rings}\label{Sec:sheaves_rings}

Let $X$ be a topological space. Denote by $\mathcal{O}(X)$ the poset category of open subsets of $X$. A sheaf on $X$ is a presheaf $F:\mathcal{O}(X)^{op} \to Set$ such that for any open cover $U = \bigcup_{i \in I}U_i$ we have: if $\{s_i \in F(U_i)\}$ is a family such that $(s_i)|_{U_i \cap U_j} = (s_j)|_{U_i \cap U_j}$, for all $i,j \in I$, then for each $i \in A$ there is a unique $s \in F(U)$ such that $s|_{U_i} = s_i$.

For a commutative ring with unity $R$ denote $\mathcal{I}(R)$ the poset category of ideals of $R$. Now, instead of open sets $U$ we will have ideals $\mathscr{I}$ and instead of intersection of open subsets we take the multiplication of ideals. In this context, we define a sheaf on $R$ as a presheaf $F:\mathcal{I}(R)^{op} \to Set$ such that for any cover $\mathscr{I} = \sum_{i \in I}\mathscr{I}_i$ we have: if $\{s_i \in F(\mathscr{I}_i)\}$ is a family such that $(s_i)|_{\mathscr{I}_i \centerdot \mathscr{I}_j} = (s_j)|_{\mathscr{I}_i \centerdot \mathscr{I}_j}$, for all $i,j \in A$, then for each $i \in I$ there is a unique $s \in F(I)$ such that $s|_{\mathscr{I}_i} = s_i$.

If we want to describe the above two notions of sheaves using the same framework, we will need quantales. 
\begin{definition}
A \textbf{quantale} $Q$ is complete lattice equipped with an associative operation $\odot: Q \times Q \to Q$, called multiplication, such that 
\begin{align*}
    a \odot (\bigvee\limits_{i \in I}b_i) = \bigvee\limits_{i \in I}(a \odot b_i) \enspace \mbox{ and } \enspace
    & (\bigvee\limits_{i \in I}b_i)  \odot a = \bigvee\limits_{i \in I}(b_i \odot a)
\end{align*}
for all $a\in  Q$ and $\{b_i\}_{i\in I} \subseteq Q$.
\end{definition}

When $\odot$ is the meet $\wedge$, the quantale is a locale. Note that $\mathcal{O}(X)$ is a locale where the order is given by the inclusion, the binary meet is the intersection, and the joins are unions.  The poset $\mathcal{I}(R)$ is a quantale with order given by the inclusion, binary meets by intersections, joins by sums, and multiplication by multiplication of ideals. These quantales share the property of being \textbf{semicartesian}, that is, for any $a, b \in Q$, $a\odot b \leq a,b$. 

In \cite{luiza2024sheaves}, we proposed the following definition of a sheaf on a semicartesian quantale: 
\begin{definition}\label{sheaf_as_functor}
A presheaf $F: Q^{op}\to Set$ is a \textbf{sheaf} if for any cover $u = \bigvee\limits_{i\in I}u_i$ of any element $u \in Q$, the following diagram is an equalizer
\begin{center}
     \begin{tikzcd}
F (u) \arrow[r, "e"] & \prod\limits_{i\in I}F (u_i) \arrow[r, "p", shift left=1 ex] 
\arrow[r, "q"', shift right=0.5 ex]  & {\prod\limits_{(i,j) \in I \times I}F (u_i \odot u_j)}
\end{tikzcd}
 \end{center}
  where:
 \begin{enumerate}
     \item $e(t) = \{t_{|_{u_i}} \enspace | \enspace i \in I\}, \enspace t \in F (u)$ 
     \item     $p((t_k)_{ k \in I}) = (t_{i_{|_{u_i \odot u_j}}})_{(i,j)\in I\times I}$ \\ $q((t_k)_{k \in I}) = (t_{j_{|_{u_i \odot u_j}}})_{(i,j)\in I\times I}, \enspace (t_k)_{k \in I} \in \prod\limits_{k\in I}F (u_k)$
 \end{enumerate}
\end{definition}

If $Q = \mathcal{I}(X)$ this is precisely the notion of sheaves on $R$ that we presented above. If $Q$ is a locale, then we obtain the well-known notion of sheaves on a locale. In particular, for $Q = \mathcal{O}(X)$ this is precisely the usual notion of sheaves on a topological space $X$, which we also described above.

Note that we could have that both the multiplication and the meet distribute over the arbitrary joins, when this happens we say that the quantale has a localic structure. 



\section{Base Change}\label{sec:base_change}

\begin{definition}
A geometric morphism is a pair of adjoint functors 
\begin{tikzcd}
	{\mathcal{Q}} & {\mathcal{Q}'}
	\arrow["{f_*}"', shift right=1, from=1-1, to=1-2]
	\arrow["{f^*}"', shift right=1, from=1-2, to=1-1]
\end{tikzcd}  such that
\begin{enumerate}
    \item $f^*$ preserves arbitrary sups and 1;
    \item $f^*$  weakly preserves the  multiplication, i.e., $f^*(p)\odot f^*(q) \leq f^*(p\odot' q), \forall p,q \in \mathcal{Q}$. 
\end{enumerate}
\end{definition}

\begin{remark}
In \cite{rosenthal_quantales} morphisms of quantales are defined as maps that preserve arbitrary sup and the multiplication.
\end{remark}
\begin{remark}
The right adjoint of $f^*$ comes from the Adjoint Functor Theorem: since $f^*$ preserves sups, it has a right adjoint $f_* : \mathcal{Q} \to \mathcal{Q}'$. We call $f_*$ direct image, and $f^*$ is called inverse image.
\end{remark}

Since $f_*$ has a left adjoint, the Adjoint Functor Theorem guarantees that $f_*$ preserves all limits. In particular, $f_*$ preserves meets and $1$. 

Even though we are considering semicartesian quantales instead of idempotent quantales, the above definition is exactly the same of a morphism of quantales given in \cite{borceux1986quantales}. As stated there, if the quantale is a locale this definition coincides with the classical definition of a morphisms of locales: since $f^*$ preserves arbitrary sups, it is a functor; consequently, we have $f^*(p \wedge q) \leq f^*(p)$ and $f^*(p \wedge q) \leq f^*(q)$. Therefore, $f^*(p \wedge q) \leq f^*(p) \wedge f^*(q)$. The definition provides the other side of the inequality.

\begin{example}
The inclusion $Idem(Q) \to Q$ is a geometric morphism with right adjoint given by $(-)^-: Q \to Idem(Q)$, where $ q^{-}   := \bigvee \{p \in Idem(Q) : p\leq q \odot p\} $ is what we call an idempotent approximation and $Idem(Q)$ is the locale of idempotent elements of $Q$. This was proved for commutative (and semicartesian) quantales in \cite[Proposition 3]{luiza2024sheaves}.
\end{example}

Next, we define what is called a \textit{strict morphism of quantales} in \cite{borceux1986quantales}.

\begin{definition}
A \textbf{strong geometric morphism of quantales} is a geometric morphism of quantales where $f^*$ preserves the multiplication. In other words, the other inequality holds. 
\end{definition}

\begin{example} \label{strgeomorp-ex}
\begin{enumerate}
    \item The inclusion $Idem(Q) \to Q$.
    \item The idempotent approximation $(-)^-: Q \to Idem(Q)$ is a strong geometric morphism if $Q$ is a geometric quantale, that is, if
$\bigvee_{i\in I} q_i^- = (\bigvee_{i\in I} q_i)^-$, for each $\{q_i : i \in I\} \subseteq Q$. 
    \item A projection map $p_i : \prod_{i\in I}Q_i   \to Q_i$ preserves sups, unit and multiplication.
    \item The inclusions $[a^-, \top] \to Q$ preserves sups, unit and multiplication, for all $a \in Q$.
    \item Every surjective 
    homomorphism $f: R \to S$ of commutative and unital rings induces a strong geometric morphism $f^* : \mathcal{I}(R) \to  \mathcal{I}(S)$ given by $f^*(J) = \langle f(J) \rangle$, the ideal generated by $f(J)$. Indeed, notice that $1 = R$ in $\mathcal{I}(R)$ and the surjectivity gives that $\langle f(R)\rangle = \langle S \rangle,$ so $f^*$ preserves $1$. To show it preserves arbitrary sups and the multiplication, we do not need the surjective condition. In fact, this is stated in  \cite[2.3(3)]{rosenthal_quantales}.  
\end{enumerate}

\end{example}

In sheaf theory the sheafification functor is a crucial construction. It is the exact left adjoint functor $a :PSh(C) \to Sh(C,J)$ of the inclusion functor $Sh(C,J) \to PSh(C)$, where $C$ is a category equipped with a Grothendieck topology $J$. Moreover, $a$ preserves all finite limits. If $D$ is a reflective subcategory of a category of presheaves in which the left adjoint functor of the inclusion preserves all finite limits, then $D$ is a Grothendieck topos. In \cite[Chapter 4.1]{tenório2024grothendieck},  we presented a general framework to show that $Sh(Q)$ is a reflective subcategory of $PSh(Q)$ (it follows from $Sh(Q)$ being a $\lambda$-orthogonality class of $PSh(Q)$, for a regular cardinal $\lambda$). Since $Sh(Q)$ may not be a Grothendieck topos \cite{luiza2024sheaves}, the sheafification in our context may not preserve all finite limits. In the next result we will use this weaker sheafification functor $a : PSh(Q) \to Sh(Q)$. The fact that it is left adjoint to the inclusion from sheaves to presheaves is the only property that we need.

\begin{theorem}\label{theorem:change_of_basis}
A geometric morphism $f: Q \to Q'$ induces an 
adjunction in the respective category of sheaves. More precisely, the pair of adjoint functors 
\begin{tikzcd}
	{\mathcal{Q}} & {\mathcal{Q}'}
	\arrow["{f_*}"', shift right=1, from=1-1, to=1-2]
	\arrow["{f^*}"', shift right=1, from=1-2, to=1-1]
\end{tikzcd} induces a pair 
\begin{tikzcd}
	{Sh(\mathcal{Q})} & {Sh(\mathcal{Q}')}
	\arrow["{\phi_*}"', shift right=1, from=1-1, to=1-2]
	\arrow["{\phi^*}"', shift right=1, from=1-2, to=1-1]
\end{tikzcd} where $\phi^*$ is left adjoint to $\phi_*$.
\end{theorem}
\begin{proof}
\textit{First step: find a $\phi_*$.}  

Define $\phi_* : Sh(\mathcal{Q}) \to Sh(\mathcal{Q}')$ by $\phi_*(F) = F \circ f^*$ and $\phi_*(\eta_u) = \eta_{f^*_u}$ for all $u \in \mathcal{Q}'$. We have to show that $F \circ f^*$ is a sheaf in $\mathcal{Q}'$. Take $u = \bigvee\limits_{i \in I }u_i$ a cover in $\mathcal{Q}'$ and a compatible family $(s_i \in F\circ f^*(u_i))_{i\in I}$ in $F\circ f^*$. This compatible family can be written as $(s_i \in F\circ f)(u_i \odot' u_j)F(f^*(u_i)))_{i\in I}$ and it remains as a compatible family in $F$ because
\begin{enumerate}
    \item  $f^*(u) = f^*(\bigvee\limits_{i \in I }u_i) = \bigvee\limits_{i \in I }f^*(u_i)$ is a cover;
    \item Since $f^*(p)\odot f^*(q) \leq f^*(p\odot' q), \forall p,q \in Q'$, we have that  ${s_i}_{|_{u_i \odot' u_j}} = {s_j}_{|_{u_i \odot' u_j}}  $ in $F\circ f)(u_i \odot' u_j)$ implies $${s_i}_{|_{f^*(u_i) \odot f^*(u_j)}}  = {s_j}_{|_{f^*(u_i) \odot f^*(u_j)}}$$
    in $ F(f(u_i) \odot f(u_j)).$
\end{enumerate}

By the sheaf condition on $F$, there is a unique $s \in F(f^*(u))$ such that $s_{|_{f^*(u_i)}} = s_i$, for all $i \in I$ in $F$. So, in $F \circ f^*$, $s_{|_{u_i}} = s_i$, for all $i \in I$, as desired.

\textit{Second step: find $\phi^*$ left adjoint to $\phi_*$.} 

Have the following diagram in mind: 
\[\begin{tikzcd}
	{PSh(\mathcal{Q}')} & {PSh(\mathcal{Q})} \\
	{Sh(\mathcal{Q}')} & {Sh(\mathcal{Q})}
	\arrow["{\phi^*}", shift left=1, from=2-1, to=2-2]
	\arrow["i", shift left=2, from=2-2, to=1-2]
	\arrow["a", shift left=1, from=1-2, to=2-2]
	\arrow["b", shift left=2, from=1-1, to=2-1]
	\arrow["j", shift left=1, from=2-1, to=1-1]
	\arrow["{\phi^-}", shift left=1, from=1-1, to=1-2]
	\arrow["{\phi^+}", shift left=1, from=1-2, to=1-1]
	\arrow["{\phi_*}", shift left=1, from=2-2, to=2-1]
\end{tikzcd}\]

where $i$ and $j$ are inclusions (reflections) with left adjoints (the sheafifications) $a$ and $b$, respectively. The functor $\phi^+$ is the precomposition with $f^*$, as we did to define $\phi_*$, and we take $\phi^- = Lan_{f^*}P$, i.e.,  $\phi^-$ is the left Kan extension of a presheaf $P: \mathcal{Q}' \to Set$ along $f^*$. Since $\phi^+$ is precomposition with $f^*$, $\phi^-$ is left adjoint to  $\phi^+$.

Define $\phi^* = a \circ \phi^- \circ j$. For any $F$ sheaf in $Sh(\mathcal{Q})$ and any $G$ sheaf in  $Sh(\mathcal{Q}')$ we have:
\begin{align*}
    Hom_{Sh(\mathcal{Q})}(a \circ \phi^- \circ j(G),F) &\cong  Hom_{PSh(\mathcal{Q})}( \phi^- \circ j(G),i(F)) 
    \\
    &\cong Hom_{PSh(\mathcal{Q}')}(j(G),\phi^+\circ i(F)) 
    \\
    & \cong Hom_{PSh(\mathcal{Q}')}(j(G),j\circ \phi^*(F)) 
    \\
    &\cong Hom_{Sh(\mathcal{Q}')}(G,b \circ j \circ \phi_*(F)) 
    \\
    &\cong Hom_{Sh(\mathcal{Q}')}(G,\phi_*(F)) 
\end{align*}

So $\phi^*$ is left adjoint to $\phi_*$.

\end{proof}

We have sufficient conditions for an equivalence:
\begin{proposition}
Consider a pair of adjoint functors 
\begin{tikzcd}
	{Q} & {Q'}
	\arrow["{f_*}"', shift right=1, from=1-1, to=1-2]
	\arrow["{f^*}"', shift right=1, from=1-2, to=1-1]
\end{tikzcd} that induces the adjunction \begin{tikzcd}
	{Sh(Q)} & {Sh(Q')}
	\arrow["{\phi_*}"', shift right=1, from=1-1, to=1-2]
	\arrow["{\phi^*}"', shift right=1, from=1-2, to=1-1]
\end{tikzcd} where $\phi^*$ is left adjoint to $\phi_*$. 
Suppose that any $u \in Q$ is of the form $u = f^*(u')$ for some $u' \in Q'$, then $\phi_*$ is a full and faithful functor. Moreover, if $f_*$ preserves sup and multiplication, and $f_* \circ f^* = id_{Q'}$, then $\phi_*$ is dense and therefore it gives an equivalence between $Sh(Q)$ and $Sh(Q')$.
\end{proposition}

\begin{proof}
We begin by showing that $\phi_{*_{F,G}} : Hom_{Sh(Q)}(F,G) \to Hom_{Sh(Q')}(F\circ f^*,G\circ f^*)$ is surjective: Let $\psi :F\circ f^* \to G\circ f^* $ be a natural transformation. For each $u' \in Q'$, $\psi_{u'}: F(f^*(u')) \to G(f^*(u'))$, which is a natural transformation $\varphi_{f^*(u)} :  F(f^*(u')) \to G(f^*(u'))$. So $\psi = \phi_*(\varphi)$.

Now, whe check that $\phi_{*}$ is faithful: Take $\varphi, \psi \in Hom_{Sh(Q)}(F,G)$ such that $\phi_{*_{F,G}}(\varphi)_{u'} = \phi_{*_{F,G}}(\psi)_{u'}$, for all $u' \in Q'$. Then $\varphi_{f^*(u')} = \psi_{f^*(u')}$. Therefore, $\varphi_u = \psi_{u}$, for all $u \in Q$, as desired.

To prove that $\phi_*$ is dense we need to use $f_*$ and suppose that it preserves sup and multiplication.

    Let $F$ be a sheaf on $Q'$. Note that $F \circ f_*$ is a sheaf on  $Q$: take $u = \bigvee_{i\in I} u_i$ a covering in $Q$. Then $\bigvee f_*(u_i) = f_*(u)$ is a covering.  
Let $(s_i \in F\circ f_*(u_i))_{i\in I}$ be a compatible family in $F \circ f_*$. So $s_{i_{|_{u_i \odot u_j}}} = s_{j_{|_{u_i \odot u_j}}}, \forall i,j \in I$. This implies that $$s_{i_{|_{f_*(u_i) \odot f_*(u_j)}}} = s_{i_{|_{f_*(u_i \odot u_j)}}} = s_{j_{|_{f_*(u_i \odot u_j)}}} = s_{j_{|_{f_*(u_i) \odot f_*(u_j)}}} $$
Therefore, this is a compatible family in $F$. Since $F$ is a sheaf, there is a unique $s \in F((f_*)(u))$ such that $s_{|_{f_*(u_i)}} = s_i,$ for all $i \in I.$ So, in $F \circ f_*$, $s_{|_{u_i}} = s_i$, for all $i \in I$. 

Finally, observe that $\phi_*(F\circ f_*)(u) = (F\circ f_*\circ f^*)(u') \cong F(u') $, for all $u' \in Q'$. So $\phi_*$ is dense.
\end{proof} 
\begin{remark}
    The extra conditions provide that $f^*$ is a homomorphism of quantales (preserve sups and multiplication) and then $f_*\circ f^* = id_{Q'}$ implies that $Q$ and $Q'$ are isomorphic quantales. In other words, in the conditions of the above theorem, it is expected to obtain an equivalence between $Sh(Q)$ and $Sh(Q')$.
\end{remark}

The importance of the above proposition is shown by the following applications -- the first and the last are already known:

\begin{itemize}
    \item If $f: X \to Y$ is a homeomorphism of topological spaces, then $f^*(U) = f(U)$ and $f_*(V) = f^{-1}(V)$ satisfy all the required conditions. So $Sh(X)$ is equivalent to $Sh(Y)$.
    \item If $f: R \to S$ is an isomorphism of commutative rings, then $f^*(I) = f(I)$ and $f_*(J) = f^{-1}(J)$ satisfy all the required conditions. So $Sh(R)$ is equivalent to $Sh(S)$.

    \item Any isomorphism $f^* : Q' \to Q$ between quantales satisfies the hypothesis. In particular, since there is an isomorphism $\mathbb{L} = ([0,\infty],+,\geq) \to  ([0,1,.,\leq]) = \mathbb{I}$ via the map $x \mapsto e^{-x}$, the categories $Sh(\mathbb{L})$ and $S(\mathbb{I})$ are equivalent.

    \item  Another interesting case of quantalic isomorphism is the localic isomorphism $D: Rad(R) \to \mathcal{O}(Spec(R))$, $I \mapsto D(I) =  \{J\,:\, J \mbox{ prime ideal of } R, I \nsubseteq J\} = \bigcup_{a \in I}D(a)$, where $D(a) = \{J\,:\, J \mbox{ prime ideal of } R, a \notin J\}$  \cite[Propostion 2.11.2]{borceux1994handbook3}. So  $Sh(\mathcal{O}(Spec(R)))$ and $Sh(Rad(R))$ are equivalent categories.
\end{itemize}

\begin{proposition}\label{prop:changebase_ring_continuous}
Let $X$ be a topological space that admits partition of unity subordinate to a cover (for example, paracompact Hausdorff spaces), and $C(X)$ the ring of all real-valued continuous functions on $X$. Then
\begin{enumerate}
    \item The functor 
\begin{align*}
  \tau\colon \mathcal{I}(C(X))  & \to \mathcal{O}(X) \\
  I & \mapsto
  \bigcup_{f \in I}f^{-1}(\mathbb{R}- \{0\})
\end{align*}
preserves arbitrary supremum, multiplication, and unity.

    \item The functor 
\begin{align*}
  \theta\colon \mathcal{O}(X)  & \to \mathcal{I}(C(X)) \\
  U & \mapsto
  \langle\{f \, : \, f\restriction_{X-U} \equiv 0 \}\rangle
\end{align*}
preserves arbitrary supremum
and unity.

    \item  The functor $\tau$ is left adjoint to $\theta$. 
\end{enumerate}
\end{proposition}
\begin{proof}
\begin{enumerate}
    \item 

    To check that $\tau$ preserves supremum, observe that if $x \in (f_1 + ... + f_n)^{-1}(\mathbb{R}- \{0\})$ then $f_j(x) \neq 0$ for some $j \in \{1,...,n\}$ and so $x \in \bigcup_{j \in \{1,...,n\}}f_j^{-1}(\mathbb{R}- \{0\})$. Therefore:

    $$\tau(\Sigma_{j\in J}I_j) = \bigcup_{f \in \Sigma_{j \in J} I_j} f^{-1}(\mathbb{R}- \{0\}) \subseteq \bigcup_{j \in J}\bigcup_{f_j \in I_j}f_j^{-1}(\mathbb{R}- \{0\})  =\bigcup_{j\in J}\tau(I_j) $$

    Since $\tau$ is increasing (if $I \subseteq K$ and $x \in \bigcup_{f \in I}f^{-1}(\mathbb{R}- \{0\}) = \tau(I)$ then $f \in K$ and $x \in \tau(J)$),  we have that $\tau(I_j) \subseteq \tau(\Sigma_{j\in J}I_j)$ for all $j \in J$. By definition of supremum, $\bigcup_{j\in J}\tau(I_j) \subseteq \tau(\Sigma_{j\in J}I_j)$.


    The fact that  $\tau$ is increasing also gives that given ideals $I$ and $J$, $\tau(IJ) \subseteq \tau(I) \cap \tau(J)$. On the other side, if $x \in  \tau(I) \cap \tau(J)$ then  $f(x) \neq 0$ and $g(x) \neq 0$, for $f \in I$ and $g \in J$. Then $fg(x) \neq 0$. Since $fg \in IJ$, we obtain that $x \in \tau(IJ)$. Therefore, $\tau$ preserves multiplication. 

    We also have $\tau(C(X)) = \bigcup_{f \in C(X)} f^{-1}(\mathbb{R}- \{0\}) = X$, since for all $x \in X$ there is some continuous function $f$ such that $f(x) \neq 0$, so $\tau$ preserves unity.

    \item Observe that 
    \begin{align*}
        \theta(\bigcup_{i\in I}U_i) &=  \langle\{f \, : \, f\restriction_{X-\bigcup_{i\in I}U_i} \equiv 0 \}\rangle
    \end{align*}
    
    Let $g \in \Sigma_{i\in I}\langle\{f \, : \, f\restriction_{X-U_i} \equiv 0 \}\rangle$. So $g = \Sigma_{i\in I}f_i$ where $f_i(x) = 0$ for all $x \notin U_i$. So $g(x) = \Sigma_{i\in I}f_i(x) = 0$ for all $x \notin U$. Actually, since we are dealing with an ideal generated by a set, each $f_i$  is a sum $\Sigma_{j\in \mathbb{N}}(\phi_{ij}.h_{ij})(x)$, where $\phi_{ij} \in C(X)$ and $h_{ij}(x) = 0$ for all $x \notin U_i$, but we write as above since it does not change the verification and it is easier to follow the argument. 

    For the other inclusion, if $U = \bigcup_{i\in I}U_i$ is a covering and $g \in \langle\{f \, : \, f\restriction_{X-U} \equiv 0 \}\rangle$, define $g_i = g . \sigma_i$ where $\{\sigma_i\}_{i\in I}$ is partition of unity subordinate to $\{U_i\}_{i\in I}$. Then $\sigma_i(x) = 0$ for all $x \notin U_i$, and thus $g_i(x) = 0$ for all $x \notin U_i$ and $$g(x) = g(x). 1 = g(x) . \Sigma_{i\in I}\sigma_i(x) = \Sigma_{i\in I}g_i(x) \in \Sigma_{i\in I}\langle\{f \, : \, f\restriction_{X-U_i} \equiv 0 \}\rangle.$$

    Note that $\theta$ preserves unity because every function in $C(X)$ when restricted to the empty-set is the empty function, and $f\restriction_{\emptyset} = 0$ vacuously holds. Thus, $\theta(X) = C(X)$

    \item  Suppose $\tau(I) \subseteq U$. Let $g \in I$, then $g^{-1}(\mathbb{R}- \{0\}) \subseteq \bigcup_{g \in I}g^{-1}(\mathbb{R}- \{0\}) \subseteq U$.

    If $x \notin U$, then $x \notin g^{-1}(\mathbb{R}- \{0\})$, thus $g(x) = 0$, for all $x \notin U$. So $g \in \theta(U).$

    Suppose $I \subseteq \theta(U)$. Let $x \in \tau(I)$, then $f(x) \neq 0$ for some $f \in I \subseteq \theta(U)$, which implies that $f(x) = 0,$ for all $x \notin U$, since $f(x) = \Sigma_{j \in J}\phi_j h_j(x)$ where $\phi_j \in C(X)$ and $h_j(x) = 0$ for all $x \notin U$. To avoid the contradiction $0 = f(x) \neq 0$, we have that $x \in U.$

\end{enumerate}
    
\end{proof}

Since every smooth manifold $M$ admits partition of unity subordinate to
a cover (second countable, Hausdorff and locally Euclidian implies paracompactivity), if $C^{\infty}(M)$ is the ring of all real-valued
smooth functions on $M$ the above proof and shows that:

\begin{proposition}
    Let $M$ be a smooth manifold, then
\begin{enumerate}
    \item The functor 
\begin{align*}
  \tau\colon \mathcal{I}(C^{\infty}(M))  & \to \mathcal{O}(M) \\
  I & \mapsto
  \bigcup_{f \in I}f^{-1}(\mathbb{R}- \{0\})
\end{align*}
preserves arbitrary supremum, multiplication, and unity.

    \item The functor 
\begin{align*}
  \theta\colon \mathcal{O}(M)  & \to \mathcal{I}(C^{\infty}(M) \\
  U & \mapsto
  \langle\{f \, : \, f\restriction_{X-U} \equiv 0 \}\rangle
\end{align*}
preserves arbitrary supremum
and unity.

    \item  The functor $\tau$ is left adjoint to $\theta$.
\end{enumerate}
\end{proposition}


\begin{proposition}\label{prop:both_precomp_preservesheaves}
 In the same conditions and notation of Proposition \ref{prop:changebase_ring_continuous}, we have that $- \circ \tau$  and    $- \circ \theta$ preserve sheaves.
\end{proposition}
\begin{proof}
    In the above proposition we proved that $\tau$ is the left adjoint in a geometric morphism given by the pair $\tau$ and $\theta$. By Theorem \ref{theorem:change_of_basis}, we have that if $F$ is a sheaf on $\mathcal{O}(X)$, then $F \circ \tau$ is a sheaf on $\mathcal{I}(C(X))$.

    Now, note that $\theta(U_i) \odot \theta(U_j) \subseteq \theta(U_i \cap U_j)$, since $g \in \theta(U_i) \odot \theta(U_j)$ is of the form $(\sum\limits_{k}\phi_k p_k)(\sum\limits_{l}\psi_l q_l)$ where $p_k(x) = 0, \forall x \notin U_i $ and $q_l(x) = 0, \forall x \notin U_j$. So $p_kq_l(x) = 0, \forall x \in (X \setminus U_i) \cup (X \setminus U_j) =  X \setminus(U_i \cap U_j)$.

    Thus, if $F$ is a sheaf on $\mathcal{I}(C(X))$ and $s_i \in  F \circ \theta(U_i)$ is such that $s_i|_{U_i \cap U_j} = s_j|_{U_i \cap U_j}$, for all $i, j \in I$, then $s_i|_{\theta(U_i) \odot \theta(U_j)} = s_j|_{\theta(U_i) \odot \theta(U_j)}$. So $s_i$ is a compatible family in $F$. Since $F$ is a sheaf, it admits unique gluing. Thus $F \circ \theta$ is a sheaf on $\mathcal{O}(X)$.
\end{proof}

Observe the following diagram 

\[\begin{tikzcd}[ampersand replacement=\&]
	{\mathcal{O}(X)} \& {\mathcal{I}(C(X))} \& {\mathcal{O}(X)} \\
	{Sh(X)} \& {Sh(C(X))} \& {Sh(X)} \\
	{PSh(X)} \& {PSh(C(X))} \& {PSh(X)}
	\arrow["\theta", from=1-1, to=1-2]
	\arrow["\tau", from=1-2, to=1-3]
	\arrow[hook, from=2-1, to=3-1]
	\arrow[hook, from=2-2, to=3-2]
	\arrow[hook, from=2-3, to=3-3]
	\arrow[shift right=1, from=2-1, to=2-2]
	\arrow[shift right=1, from=2-2, to=2-3]
	\arrow["{Lan_\theta}"', shift right=1, from=3-1, to=3-2]
	\arrow["{- \circ \tau}"', shift right=1, from=2-3, to=2-2]
	\arrow["{- \circ \theta}"', shift right=1, from=2-2, to=2-1]
	\arrow["{- \circ \theta}"', shift right=1, from=3-2, to=3-1]
	\arrow["{Lan_\tau}"', shift right=1, from=3-2, to=3-3]
	\arrow["{- \circ \tau}"', shift right=1, from=3-3, to=3-2]
\end{tikzcd}\]

Take $P$ a presheaf in $PSh(X)$, the left Kan extension of $P$ along $\theta$ is the following colimit (we are using the pointwise Kan extension  formula but applied to $\mathcal{O}(X)^{op}$ and $\mathcal{I}(C(X))^{op}$, since presheaves are contravariant functors)
\begin{align}\label{Lan_preservessheaves}
    (Lan_{\theta}P)(I) = \varinjlim\limits_{\theta(U) \supseteq I} P(U) = P(\tau(I)) = (P\circ \tau)(I) 
\end{align}
because $\tau$ is left adjoint to $\theta$. Since $\tau$ is a geometric morphism (Proposition \ref{prop:changebase_ring_continuous}), by Theorem \ref{theorem:change_of_basis} we have that $P\circ \tau$ is a sheaf. So $Lan_{\theta}$ also preserves sheaves. 

 Since the left Kan extension along $\theta$ is left adjoint to the precomposition with $\theta$, we have that $(- \circ \tau)$ is left adjoint to $(- \circ \theta)$. Let $P$ presheaf in $PSh(X)$ and $F$ sheaf in $Sh(C(X))$  
\begin{align*}
    Hom_{Sh(C(X))}(a_{C(X)} \circ (-\circ \tau)(P),F) &\cong   Hom_{PSh(C(X))}( (-\circ \tau)(P), j(F)) \\
    &\cong Hom_{PSh(X)}(P,(-\circ \theta)\circ j(F)) 
\end{align*}
By Proposition \ref{prop:both_precomp_preservesheaves}, $(- \circ \theta)$  preserves sheaves. Then we have
\begin{align*}
    Hom_{Sh(C(X))}((-\circ \tau)\circ a_X(P),F) &\cong   Hom_{Sh(C(X))}(Lan_\theta \circ a_X(P), F) \\
    & \cong Hom_{Sh(X)}(a_X(P), (-\circ \theta)(F))\\ &\cong Hom_{PSh(X)}(P,i \circ (-\circ \theta)(F)) 
\end{align*}

Notice that $i \circ (-\circ \theta)(F) = (-\circ \theta)\circ j(F)$. So $(-\circ \tau)\circ a_X$ and $a_{C(X)} \circ (-\circ \tau)$ are both left adjoint functors of the same functor. Therefore, $(-\circ \tau)\circ a_X \cong a_{C(X)} \circ (-\circ \tau)$ and we obtain that the following diagram commutes (up to natural isomorphism)
\[\begin{tikzcd}[ampersand replacement=\&]
	{Sh(X)} \&\&\&\& {Sh(C(X))} \\
	\\
	{PSh(X)} \&\&\&\& {PSh(C(X))} \\
	\&\& Set
	\arrow["i", shift left, hook, from=1-1, to=3-1]
	\arrow["j", shift left, hook, from=1-5, to=3-5]
	\arrow["{Lan_\theta = -\circ \tau}"', shift right, from=1-1, to=1-5]
	\arrow["{Lan_\theta = -\circ \tau}"', shift right, from=3-1, to=3-5]
	\arrow["{- \circ \theta}"', shift right, from=1-5, to=1-1]
	\arrow["{- \circ \theta}"', shift right=2, from=3-5, to=3-1]
	\arrow["{a_X}", shift left, from=3-1, to=1-1]
	\arrow["{a_{C(X)}}", shift left, from=3-5, to=1-5]
	\arrow["{const_{X}}", shift left, from=4-3, to=3-1]
	\arrow["{const_{C(X)}}"'{pos=0.7}, shift right, from=4-3, to=3-5]
\end{tikzcd}\]
Therefore, we proved:
\begin{corollary}\label{cor:constantsheafinCX}
   The constant sheaf in $Sh(C(X))$, denoted $\underline{K}^a_{C(X)}$, is naturally isomorphic to the composition $\underline{K}^a_{X} \circ \tau$ where $\underline{K}^a_{X}$ is a constant sheaf in $Sh(X)$. 
\end{corollary}
\begin{proof}
   We are using $\underline{K}^a_{X}$ to denote the sheafification of a constant presheaf $\underline{K}_{X}$. Explicitly, the commutativity of the diagram reads as follows:
   \begin{align*}
       \underline{K}^a_{C(X)} &= a_{C(X)} \circ const_{C(X)}(K) \\
       &= a_{C(X)}\circ(-\circ \tau) \circ const_{X}  (K) \\
       &\cong (-\circ \tau) \circ a_{X}\circ const_X (K) \\
       &=(-\circ \tau)(a_X(K))\\
       &= a_X(K)\circ \tau\\
       &= \underline{K}^a_{X}\circ \tau 
   \end{align*}
\end{proof}

\begin{remark}
    The above still holds for sheaves valued in the category of abelian groups and homomorphisms $Ab$ instead of the category of sets and functions $Set$.
\end{remark}

We believe that it is possible to use the above construction in other contexts. For instance, we could replace real-valued (continuous/smooth) functions with complex-valued (continuous/smooth) functions. We plan to check such change in a future work.

\section{\v{C}ech Cohomology}
Let $u$ be a cover in a quantale $Q$ and $F: Q^{op} \to Ab$ an abelian presheaf. 
\begin{definition}
The \textbf{\v{C}ech cochain complex} for $u$ with values in $F$ is given by $$C^q(u, F) = \prod_{i_0<...<i_q}F(u_{i_0} \odot \dots \odot u_{i_q})$$
 and the \textbf{coboundary morphism} $d^q: C^q(u,F)\to C^{q+1}(u,F)$ by $$(d^qi) = \sum^{q+1}_{k=0}(-1)^ki(\delta_k)|_{u_{i_0} \odot \dots \odot u_{i_q} \odot u_{i_{q+1}}},$$
where $\delta_k$ means that we removed the  $i_k$-entry.
\end{definition}

Observe that if $Q = (\mathcal{O}(X),\subseteq, \cap)$, then the definition  above is precisely the usual \v{C}ech cochain complex for a topological space $X$.

\begin{definition}
    
Given a sheaf $F$, we define the \textbf{\v{C}ech cohomology} for $u$ with values in $F$ by $$\mbox{\v{H}}^q(u, F) = \frac{Ker d^q}{Im d^{q-1}} $$
\end{definition}

This construction is about the covering and not about the actual space. When $Q = (\mathcal{O}(X),\subseteq, \cap)$, the topological space $X$ may admit many covers so this construction is not enough to talk about the cohomology of $X$, but only about the cohomology of a fixed cover of $X$. There is a way to define a \v{C}ech cohomology of $X$ and, analogously, we can define a \v{C}ech cohomology of a commutative ring with unity $R$. Before we get deeper into the theory, we explore the application we are interested in.

For convenience, we recall the pair of adjoint functors 
\begin{align*}
  \tau\colon \mathcal{I}(C(X))  & \to \mathcal{O}(X) \\
  I & \mapsto
  \bigcup_{f \in I}f^{-1}(\mathbb{R}- \{0\})
\end{align*}
and 
\begin{align*}
  \theta\colon \mathcal{O}(X)  & \to \mathcal{I}(C(X)) \\
  U & \mapsto
  \langle\{f \, : \, f\restriction_{X-U} \equiv 0 \}\rangle
\end{align*}

that we introduced in \ref{prop:changebase_ring_continuous}.

\begin{proposition}\label{prop:iso_of_cohomology_of_cover}
    Fix a cover $\mathcal{U}$ of $C(X)$. Then the \v{C}ech cohomology group for $\tau(\mathcal{U})$ with values in the constant sheaf $\underline{K}_{X}$ is isomorphic to the \v{C}ech cohomology group for $\mathcal{U}$   with values in the constant sheaf $\underline{K}_{C(X)}$. 
\end{proposition}

\begin{proof}
By Corollary \ref{cor:constantsheafinCX},
\begin{align*}
        C^q(\mathcal{U},\underline{K}_{C(X)}) &= \prod\limits_{i_0<...<i_q}\underline{K}_{C(X)}(u_{i_0}\odot...\odot u_{i_q}) \\
      &\cong \prod\limits_{i_0<...<i_q}(\underline{K}_{X}\circ \tau) (u_{i_0}\odot...\odot u_{i_q}) \\
      &=  \prod\limits_{i_0<...<i_q} \underline{K}_{X} (\tau (u_{i_0}) \cap ... \cap \tau (u_{i_q}))\\
      &= C^q(\tau(\mathcal{U}),\underline{K}_{X})     
\end{align*}

Then we have commutative squares
\[\begin{tikzcd}[ampersand replacement=\&]
	\dots \& {C^{q-1}(\mathcal{U},\mathbb{Z}_{C(X)})} \& {C^q(\mathcal{U},\mathbb{Z}_{C(X)})} \& {C^{q+1}(\mathcal{U},\mathbb{Z}_{C(X)})} \& \dots \\
	\dots \& {C^{q-1}(\tau(\mathcal{U}),\mathbb{Z}_{X})} \& {C^q(\tau(\mathcal{U}),\mathbb{Z}_{X})} \& {C^{q+1}(\tau(\mathcal{U}),\mathbb{Z}_{X})} \& \dots
	\arrow["{d^{q}}", from=1-3, to=1-4]
	\arrow["{d^{q-1}}", from=1-2, to=1-3]
	\arrow["{d^{q-1}}"', from=2-2, to=2-3]
	\arrow["{d^{q}}"', from=2-3, to=2-4]
	\arrow[from=1-4, to=1-5]
	\arrow[from=2-4, to=2-5]
	\arrow[from=1-1, to=1-2]
	\arrow[from=1-2, to=2-2]
	\arrow[from=1-3, to=2-3]
	\arrow[from=1-4, to=2-4]
	\arrow[from=2-1, to=2-2]
\end{tikzcd}\]
Thus, we have an isomorphism of cochain complexes and so the cohomology groups are isomorphic: $\mbox{\v{H}}^q(\mathcal{U}, \underline{K}_{C(X)}) \cong  \mbox{\v{H}}^q(\tau(\mathcal{U}), \underline{K}_{X})$.
\end{proof}

Since the above holds for any cover $\mathcal{U}$ of $C(X)$, we expect that something at least similar will happen between the cohomology groups of $C(X)$ and $X$. First:

\begin{definition}
    Fix a quantale $Q$ and an element $u \in Q$. Let ${\cal{U}} = (u_i)_{i \in I}$ and ${\cal{V}} = (v_j)_{j \in J}$ be coverings of $u$. We say that ${\cal{U}}$ is a \textbf{refinement} of ${\cal{V}}$ if there is a function $r : I \to J$  and a morphism $u_i \to v_{r(i)}$, for all $i \in I.$  
\end{definition}

Given $r: I \to J$ that testifies ${\cal{U}}$ as refinement of ${\cal{V}}$, we have an induced morphism of cochain complexes $m_r: C^\bullet({\cal{V}}, F) \to  C^\bullet({\cal{U}}, F)$ and a corresponding morphism of \v{C}ech cohomology groups with respect to the coverings ${\cal U}$ and ${\cal V}$, $\check{m}_r : \check{\mathrm{H}}^\bullet({\cal V},F)  \to \check{\mathrm{H}}^\bullet({\cal U},F) $. Moreover, if $s : I \to J$ is another chosen function with respect to the refinement that testifies that ${\cal{U}}$ is a refinement of ${\cal{V}}$, then the induced morphisms of complexes $m_r, m_s$ are homotopic. Therefore, there is a unique induced morphism of cohomology groups $\check{m}_{{\cal U}, {\cal V}} : \check{\mathrm{H}}^\bullet({\cal V},F)  \to \check{\mathrm{H}}^\bullet({\cal U},F) $

Besides it, the class $\mbox{Ref}(u)$ of all coverings of $u$ is partially ordered under the refinement relation; this is a directed ordering relation. Thus, we can define: 
\begin{definition}
The \textbf{\v{C}ech cohomology group for an element} $u \in Q$ with values in a sheaf $F$ is the directed (co)limit\footnote{This (co)limit has to be taken with some set-theoretical care that we do not detail.} $$\check{\mathrm{H}}^q(u,F) :=   \varinjlim\limits_{{\cal U} \in \mbox{Ref}(u)}\check{\mathrm{H}}^q({\cal U},F) .$$
\end{definition}
\begin{theorem}\label{teo:iso_of_cohomology_of_space}
  The \v{C}ech cohomology group of $C(X)$ values in  $\underline{K}_{C(X)}$ is isomorphic to  the \v{C}ech cohomology group of  $X$ with values in  $\underline{K}_{X}$. 
\end{theorem}
\begin{proof}
First, observe that by Proposition \ref{prop:iso_of_cohomology_of_cover}:
\begin{align*}
    \check{\mathrm{H}}^q(C(X),\underline{K}_{C(X)}) &:=   \varinjlim\limits_{{\cal U} \in \mbox{Ref}(C(X))}\check{\mathrm{H}}^q({\cal U},\underline{K}_{C(X)}) \\
    &\cong \varinjlim\limits_{{\cal U} \in \mbox{Ref}(C(X))}\check{\mathrm{H}}^q(\tau({\cal U}),\underline{K}_{X}) 
\end{align*} 

Now, let $\mathcal{V}  = (V_i)_{i \in I}$ be a covering of $X$. It is clear that $\theta(V_i) \subseteq \theta(V_i)$. Since $\tau$ is left adjoint to $\theta$ we obtain $\tau(\theta(V_i)) \subseteq V_i$. Recall that $\theta$ and $\tau$ preserve supremum (Proposition \ref{prop:changebase_ring_continuous}), thus $\tau(\theta(V_i))$ is a covering of $X$ and then  $\tau(\theta(\mathcal{V}))$ it is a refinement of $\mathcal{V}$. In other words, for every $\mathcal{V}$ covering of $X$ there is a covering $\mathcal{U}$ of $C(X)$ such that $\tau(\mathcal{U}) \subseteq \mathcal{V}$.  Therefore
\begin{align*}
    \check{\mathrm{H}}^q(X,\underline{K}_{X}) &:=   \varinjlim\limits_{{\cal V} \in \mbox{Ref}(X)}\check{\mathrm{H}}^q({\cal V},\underline{K}_{X}) \\
    &\cong \varinjlim\limits_{{\cal U} \in \mbox{Ref}(C(X))}\check{\mathrm{H}}^q(\tau({\cal U}),\underline{K}_{X})
\end{align*}
\end{proof}
Once we defined sheaves on quantales and desire to make the notion of sheaves on rings similar to that of sheaves on topological spaces, it was straightforward how to define the correspondent \v{C}ech cohomology groups. However, it is not simple to actually calculate  the cohomology of an arbitrary ring, mainly because of two reasons: (i) we have to choose some sheaf but finding concrete examples of sheaves on a ring was not easy. We thought  about some presheaves examples that we could sheafify, but our sheafification process is too abstract. Even the behavior of the constant sheaf was not easy to capture; (ii) given an arbitrary ring, it may be difficult to understand what to expect from the covering of its ideals and we were not able to find studies about it in the literature. 

We solved those two difficulties: on one hand note that Theorem \ref{theorem:change_of_basis} is also a machine for producing sheaves on quantales. In particular, if we have a sheaf on a locale $L$ and a pair of adjoint functors 
\begin{tikzcd}[ampersand replacement=\&]
	{Q} \& {L}
	\arrow["{f_*}"', shift right=1, from=1-1, to=1-2]
	\arrow["{f^*}"', shift right=1, from=1-2, to=1-1]
\end{tikzcd}  then  $F \circ f^*$ is sheaf on the quantale $Q$. So we can use well-known sheaves on locales to create sheaves on quantales. Moreover, we hoped that if we had a ring that comes from a topological space, then we could indirectly calculate the cohomology of the ring  by calculating the cohomology of the space, which probably is already known since  \v{C}ech cohomology of a topological space is a topic that has been studied for a longer time. On the other hand, the above theorem is an example of such a phenomenon and, surprisingly or not, the proof relied basically on finding an adjoint pair of functors between the ideals of $C(X)$ and the open subsets of $X$, not depending of the elements of $C(X)$ or points of $X$. 

Besides it, Theorem \ref{teo:iso_of_cohomology_of_space} is not only about finding ways to calculate the \v{C}ech cohomology groups of $C(X)$, it is also about relating algebraic properties of $C(X)$ and topological properties of $X$: if $X$ is a compact manifold $n$ and class at least ${\cal C}^{n+1}$ then there is an isomorphism $H^q_{dR}(M) \cong \check{H}^q(M,\mathbb{R})$, for all $q \leq m$, where $H^q_{dR}$ denotes the de Rham cohomology groups and $\mathbb{R}$ is the constant sheaf with values in $\mathbb{R}$ (this is the De Rham Theorem. See for example  \cite[Appendix]{petersen2006riemannian}). While the dimension of $H^0_{dR}(M)$ corresponds to the  number of connected components of $X$, $X$ is connected if and only if $C(X)$ only has trivial idempotent elements ($0$ and $1$).  So, we believe that the dimension of the \v{C}ech cohomology group $\check{\mathrm{H}}^0(R,\underline{K}_{R})$ is related to the number of idempotent elements of a commutative ring with unity $R$.

The idea of investigating algebraic properties of $C(X)$ by analyzing topological properties of $X$ (and vice-versa) is not new (\cite{10.2307/1990875}, \cite{hewitt1948rings}, \cite{gillman2017rings}), but the cohomological approach is not usual: in fact,  we have found only one paper with such an approach. In \cite{watts1965alexander}, Watts creates a cohomology theory for a commutative algebra over a fixed algebra such that the case of the algebra of continuous real-valued functions on a compact Hausdorff space coincides with the  \v{C}ech cohomology of the space with real coefficients (the constant sheaf with values in $\mathbb{R}$). Our approach have the following main distinctions: 
\begin{itemize}
    \item we do not have to construct a new cohomology theory, we actually are expanding \v{C}ech cohomology in a quite natural manner;  
    \item our space $X$ does not have to be compact, but $X$ needs to admits partition of unity subordinate to a cover;
    \item we provide a general framework to investigate another algebro-geometric phenomenon that relies on finding ``good'' functors between the locale of open subsets of a space and the quantale of the ideals of a ring that arises from such space. Conversely, it is also possible to use the exact idea but starting with a ring and establishing a space that arises from such ring, as the process of taking the spectrum of a ring (then check if the functors between the respective quantales are good). 
    \item  Watts says that ``it is not easy to see how to remove the restriction to real coefficients, and we have made no attempt to do so''. In our case, it is clear how to proceed to change the coefficient, as we see below.   
\end{itemize}

\begin{theorem}\label{teo:iso_of_cohomology_of_space(for_any_coeff)}
  The \v{C}ech cohomology group of  $X$ with values in  $F$ is isomorphic to  the \v{C}ech cohomology group of $C(X)$ values in $F \circ \tau$.
\end{theorem}
\begin{proof}
    Observe that 
    \begin{align*}
       C^q(\tau(\mathcal{U}),F) &=  \prod\limits_{i_0<...<i_q} F (\tau (u_{i_0}) \cap ... \cap \tau (u_{i_q}))\\
       &= \prod\limits_{i_0<...<i_q}(F\circ \tau) (u_{i_0}\odot...\odot u_{i_q}) \\
       &= C^q(\mathcal{U},F\circ \tau)
    \end{align*}
     
Then $\mbox{\v{H}}^q(\mathcal{U}, F \circ \tau ) \cong  \mbox{\v{H}}^q(\tau(\mathcal{U}), F)$.

By the same reasoning of the proof of Theorem \ref{teo:iso_of_cohomology_of_space}, $\check{\mathrm{H}}^q(C(X),F\circ \tau  ) \cong \check{\mathrm{H}}^q(X,F)$.
\end{proof}

\begin{remark}
    Note that we are using the unity/top elements $1$ and $1'$ because usually one talks about the \v{C}ech cohomology of a topological space $X$ and $X$ is the unity/top element of $(\mathcal{O}(X),\subseteq)$.
\end{remark}

This theorem gives generality but obscures how to interpret the result, since coefficients in different sheaves lead to different cohomology theories. De Rham Cohomology measures to which extent the Stokes Theorem fails; Singular homology measures the number of holes of $X$ and it is related to singular cohomology by the universal coefficient theorem for cohomology; now, without a specific sheaf in mind, is more difficult to have a first guess of what the \v{C}ech cohomology groups of a  ring are measuring.

Nevertheless, we have a result that generalizes the above phenomenon.

\begin{theorem}\label{teo:iso_cohomology_any_quantale}
    Consider a strong geometric morphism  
\begin{tikzcd}
	{(Q,\odot,1)} & {(Q',\odot',1')}
	\arrow["{f_*}"', shift right=1, from=1-1, to=1-2]
	\arrow["{f^*}"', shift right=1, from=1-2, to=1-1]
\end{tikzcd} such that $f_*$ preserves unity and arbitrary joins. Then  $\check{\mathrm{H}}^q(1',F\circ f^*  ) \cong \check{\mathrm{H}}^q(1,F)$.
\end{theorem}
\begin{proof}
Consider a covering $\{u'_i\}_{i\in I} = \mathcal{U}'$ in $Q'.$ 
Then  

\begin{align*}
        C^q(\mathcal{U}',F\circ f^*) &= \prod\limits_{i_0<...<i_q}F\circ f^*(u'_{i_0}\odot'...\odot' u'_{i_q}) \\
      &=  \prod\limits_{i_0<...<i_q} F (f^* (u'_{i_0}) \odot ... \odot f^* (u'_{i_q}))\\
      &= C^q(f^*(\mathcal{U}'),F).
\end{align*}

Then $\check{\mathrm{H}}^q(\mathcal{U}',F\circ f^*  ) \cong \check{\mathrm{H}}^q(f^*(\mathcal{U}'),F).$

Observe that $f^*(f_*(\mathcal{U}))$ is a refinement of $\mathcal{U}$, by the same argument used in Theorem \ref{teo:iso_of_cohomology_of_space}. Thus, for every covering $\mathcal{U}$ of $1$ there is a covering $\mathcal{U}'$ of $1'$ such that $f^*(\mathcal{U}') \subseteq \mathcal{U}$. Then, $\varinjlim\limits_{{\cal U} \in \mbox{Ref}(1)}\check{\mathrm{H}}^q({\cal U},F) \cong \varinjlim\limits_{{\cal U}' \in \mbox{Ref}(1')}\check{\mathrm{H}}^q({\cal U}',F\circ f^*)$, as desired to obtain the result. 
\end{proof}

Again, the conclusion is that the isomorphism between \v{C}ech cohomology group relies exclusively on the properties of the geometric morphism between the quantales. The pair 
\begin{tikzcd}[ampersand replacement=\&]
	Q \& {Idem(Q)}
	\arrow["i"', shift right=1, from=1-2, to=1-1]
	\arrow["{(-)^-}"', shift right=1, from=1-1, to=1-2]
\end{tikzcd}
where $i$ is the inclusion and $(-)^-$ is the idempotent approximation is another pair of adjoint functors that satisfies the hypothesis of the above theorem, if $\bigvee_{i\in I} q_i^- = (\bigvee_{i\in I} q_i)^-$, for each $\{q_i : i \in I\} \subseteq Q$. Recall that a surjective ring homomorphism $f: R \to S$ induces a strong geometric morphism of quantales where $f^*(J) = f(J)$ and $f_*(K) = f^{-1}(K)$ (see Example \ref{strgeomorp-ex}). Consider the quotient map $q: R \to R/I$ defined by $q(r) = r+ I$. So, it remains to prove that $q_*$ induced by the pre-image preserves unity and supremum to have another class of examples to apply the above theorem.  Observe that for any ideal $K$ of $R/I$ we have, for some $J$ ideal of $R$, $$q_*(K) = q_*(q(J)) = q_*(\{j + I\,:\, j \in J\}) = \bigcup_{j\in J}(j+I) = J+I.$$
Then, $q_*(R/I) = q_*(q^*(R)) = I + R = R$ and
\begin{align*}
    q_*(\sum_{i=1}^{n}K_i) &= q_*(\sum_{i=1}^{n}q(J_i)) = q_*(q\sum_{i=1}^{n}(J_i))\\
    &=  I + \sum_{i=1}^{n}J_i = \sum_{i=1}^{n}(I+J_i)  \\ &= \sum_{i=1}^{n}q_*(q^*(J_i)) = \sum_{i=1}^{n}q_*(K_i),
\end{align*}

as desired.

Note that if $R$ and $S$ are Morita equivalent commutative rings, then $R$ and $S$ are isomorphic, providing another case of adjoint pair between the ideals of $R$ and the ideals of $S$, and thus an isomorphism between the cohomology of $R$ and the cohomology of $S$. Thus, our cohomology admits a Morita equivalence in the trivial case of commutative rings. We believe a next interesting application is to investigate the relation between the quantales of (bilateral) ideals of Morita equivalent non-commutative rings.

\section{Conclusions and Future works}

We have found an isomorphisms between the cohomology groups of the ring $C(X)$ and the cohomology groups of the topological space $X$ by identifying a pair of (adjoint) functors between the quantale of ideals of $C(X)$ and the locale of open subsets of $X$. A first next step is to determine if a similar construction for real-valued smooth (instead of continuous) functions still yields the desired properties, i.e, if we meet the conditions of Theorem \ref{teo:iso_cohomology_any_quantale}.  Similarly, we could explore replacing real-valued with complex-valued functions.

Additionally, we are interested in better understanding the paper \cite{borceux1994generic}, for two reasons: firstly, their functor construction of radical of a quantale (that gives the notion of radical of ideals if the quantale is the formed by ideals of a ring) establishes an adjunction that appears to fit the hypothesis of Theorem    
 \ref{teo:iso_cohomology_any_quantale}. While our construction began with a topological space and associated it with a ring, we may begin with a ring and then associate it with a topological space (radical ideals form a locale that is (localic) isomorphic to the open subsets of the spectrum of a ring). This leads to another interesting isomorphism between the cohomology groups of a ring $R$ and the cohomology groups of the space $Spec(R)$. Secondly, in the paper by Borceux and Cruciani  they also develop a notion of sheaves on a quantale; however, it is similar to what we understand as a $Q$-set.  In the localic case, the category of sheaves on a locale $L$ and the category of $L$-sets whose objects -- the $L$-sets -- are given by a set $A$ with an operation $\delta : A\times A \to L$ such that $\delta(a,b) = \delta(b,a) \mbox{ and }
    \delta(a,b) \wedge \delta(b,c) \leq \delta(a,c), \forall a,b,c \in A$, and  further conditions are satisfied to account for the gluing condition, are equivalent. Because of such equivalence, the authors coherently refer to their objects of sheaves. However, now that we have a structure that more closely resembles what we typically call a sheaf, it is natural to wonder if this equivalence between the categories still holds in the quantalic case.

    The above considerations and the results we proved are indications that we have expanded \v{C}ech cohomology in a meaningful way. Our long term goal is to use such expansion to use \v{C}ech cohomology in a wider range of well-established cohomology theories. 

\section{Acknowledgments}
This paper contains part of the thesis of Ana Luiza Tenório, which was funded by Coordenação de Aperfeiçoamento de Pessoal de Nível Superior (CAPES). Grant Number 88882.377949/2019-01. The main ideas were developed during her time at the Heinrich Heine University Düsseldorf, with financial support form CAPES. Grant Number 88887.694529/2022-00.  I am grateful for the support provided by the GRK 2240 (Algebro-Geometric Methods in Algebra, Arithmetic, and Topology) group at Düsseldorf for hosting me.




\bibliographystyle{abbrv} 
\bibliography{ArXiv_sh_q}

\end{document}